\documentclass[12pt,a4paper,twoside,english]{smfart}

\usepackage{lmodern,mdframed,xcolor}
\usepackage[utf8]{inputenc}
\usepackage[francais]{babel}

\usepackage[OT2,T1]{fontenc}

\font\rurm=wncyr10 scaled \magstep1
\mdfdefinestyle{MyFrame}{%
    linecolor=blue,
    innerrightmargin=5pt,
    innerleftmargin=5pt,
    splittopskip=\topskip,
    backgroundcolor=blue!60}

\usepackage{amssymb}
\usepackage{amsmath}
\usepackage{smfthm,mathabx}
\usepackage{enumerate}
\usepackage{amsthm}
\usepackage{amsfonts}
\usepackage{graphicx}
\usepackage[colorlinks=true,linkcolor=black,citecolor=black,urlcolor=black]{hyperref}
\usepackage{fancyhdr}
\usepackage[all,cmtip]{xy}
\usepackage{alltt}
\usepackage{footmisc}
\usepackage{smfthm}

\usepackage{calrsfs}

\usepackage{fullpage}

\xyoption{all}

\xyoption{frame}

\newcommand{\finpreuve}{\mbox{} \hfill \mbox{$\Box$}}

\def\Q{{\mathbb Q}}
\def\Z{{\mathbb Z}}
\def\fq{{\mathbb F}}

\def\Reel{{\mathbb R}}

\def\p{{\mathfrak p}}
\def\P{{\mathfrak P}}

\def\q{{\mathfrak q}}

\def\A{{\mathfrak A}}

\def\d{{\partial}}

\def\N{{\rm N}}

\def\d{{\rm d}}
\def\G{{\rm G}}
\def\K{{\rm K}}
\def\L{{\rm L}}
\def\F{{\rm F}}

\def\M{{\rm M}}

\def\V{{\rm V}}

\def\X{{\rm X}}

\def\XX{{\mathcal X}}

\def\O{{\mathcal O}}

\def\U{{\mathcal U}}

\def\H{{\rm H}}

\def\Gal{{\rm Gal}}

\def\Cl{{\rm Cl}}

\def\ker{{\rm ker}}

\def\sl{{\mathfrak s} {\mathfrak l}}
\def\1{{\bf 1}}

\def\FF#1#2{{\displaystyle{\left(\frac{#1}{#2}\right)}}}

\def\sha{{{\textnormal{\rurm{Sh}}}}}
\def\CyB{{{\textnormal{\rurm{B}}}}}

\def\Sha{{\sha}^2}

\newtheorem{Theorem}{Theorem}

\newtheorem{Corollary}{Corollary}
\newtheorem*{Remark}{Remark}
\newtheorem{Example}{Example}

\author{Farshid Hajir, Christian Maire, Ravi Ramakrishna}
\address{Department of Mathematics, University of Massachussetts, Amherst, MA 01003, USA}
 \address{FEMTO-ST Institute, Universit\'e Bourgogne Franche-Comt\'e, CNRS,  15B avenue des Montboucons, 25000 Besancon, FRANCE} 
\address{Department of Mathematics, Cornell University, Ithaca, NY 14853-4201, USA}
\email{hajir@math.umass.edu, christian.maire@univ-fcomte.fr, ravi@math.cornell.edu}

\thanks{This work started when the second author held  visiting scholar position at Cornell University,   funded by the program "Mobilit\'e sortante" of the R\'egion Bourgogne Franche-Comt\'e, during the 2017-18   academic year. It has been finished  during a visit at Harbin Institute of Technology.    CM  thanks the  Department of Mathematics at Cornell University and the  Institute for Advanced Study in Mathematics of HIT for providing a beautiful research atmospheres.
The second author  was partially supported by the ANR project FLAIR (ANR-17-CE40-0012) and  by the EIPHI Graduate School (ANR-17-EURE-0002). The third author was supported by Simons Collaboration grant \#524863. All three authors were supported by Mathematisches Forschunginstitut Oberwolfach for a Research in Pairs visit in January, 2019. }

\begin{document}

\date{\today}

\title{On the Shafarevich group of restricted ramification extensions of number fields in the tame case}

\maketitle

\subjclass{ }

\begin{abstract}
Let $\K$ be a number field and $S$ a  finite set of places of $\K$.
We study the kernels $\Sha_S$ of maps $H^2(\G_S,\fq_p) \rightarrow \oplus_{v\in S} H^2(\G_v,\fq_p)$. 
There is a natural injection $\Sha_S \hookrightarrow \CyB_S$, 
into  the dual $\CyB_S$ of a certain readily computable Kummer group $\V_S$, 
 which is always an isomorphism in the wild case. 
The tame case is much more mysterious. Our main result is that given a finite  $X$ coprime to $p$, there
exists a finite set of places $S$ coprime to $p$ such that
$\Sha_{S\cup X} \stackrel{\simeq}{\hookrightarrow} \CyB_{S\cup X} \stackrel{\simeq}{\twoheadleftarrow}
 \CyB_X \hookleftarrow \Sha_X  $. In particular, we show that in the tame case $ \Sha_Y$ can {\it increase} with increasing $Y$.
 This is in  contrast with the wild case where $\Sha_Y$ is nonincreasing in size with increasing $Y$.
\end{abstract}


Let $\K$ be a number field, and let  $S$ be a finite set of  places  of $\K$. 
Denote by  $\K_S$  the maximal extension of~$\K$ unramified outside $S$, and set $G_S=\Gal(\K_S/\K)$.
Given a prime number $p$, let  $\Sha_S$ be the $2$-Shafarevich group associated to $\G_S$ and $p$: it is the kernel of the localization map of the  cohomology group $H^2(\G_S,\fq_p)$:
 $$\Sha_S:=\Sha(\G_S,\fq_p)= \ker\big(H^2(\G_S,\fq_p)\rightarrow \oplus_{v\in S} H^2(\G_v,\fq_p)\big),$$
where  $\G_S$ acts trivially on $\fq_p$.  We  denote by $\G_v$ the absolute Galois group of the maximal extension of the completion $\K_v$ of $\K$ at $v$.

\medskip

It is well-known that $\Sha_S$  is closely related to $\CyB_S =\left( \V_S/\K^{\times p}\right)^\vee$, where
$$\V_S=\{ x \in \K^\times, v(x)\equiv 0 \  ({\rm mod} \ p) \ \forall v; x \in \K_v^p  \  \forall v\in S  \}.$$
Clearly $\K^{\times p} \subset \V_S$ and $S \subset T \implies \V_T \subset \V_S$.
Namely, in the wild case, when $S$ contains all the places above~$p$ and all archimedean places, by Poitou-Tate duality Theorem one has $\Sha_S \simeq \CyB_S$. See for example \cite[Chapter X, \S 7]{NSW}.
It is important to note that algorithms exist to compute $\CyB_S$ via ray class group computations over $\K(\mu_p)$, so in the wild case one can,
 at least in theory,
 compute $d_p \Sha_S$. 
For the more general tame situation, one only has the following injection (due to Shafarevich and Koch, see for example \cite[Chapter 11, \S 2]{Koch} or \cite[Chapter 10, \S 7]{NSW}) 
\begin{eqnarray} \label{equation:0}
 \Sha_S \hookrightarrow \CyB_S.
 \end{eqnarray}
{\it  At present there is no general algorithm    to compute $d_p \Sha_S$ in the tame case,  short of computing $\G_S$ itself.}

 \medskip

Let us write $\K_S(p)/\K$ as the maximal pro-$p$ extension of $\K$ inside $\K_S$, and put $\G_S(p)=\Gal(\K_S(p)/\K)$.
It is an exercise to see the quotient $\G_S \twoheadrightarrow \G_S(p)$ induces the injection $\Sha_{S,p} \hookrightarrow \Sha_S$,
where $\Sha_{S,p}:=\ker \left( H^2 (\G_S(p),\fq_p) \rightarrow \oplus_{v\in S} H^2(\G_v,\fq_p)\right)$.
Observe that we can take $\G_v(p)$ instead of $\G_v$,  due to the fact that $H^2(\G_v(p),\fq_p) \simeq H^2(\G_v,\fq_p)$ 
(see for example \cite[Chapter VII, \S 5]{NSW}).

The  Shafarevich group $\Sha_S$ is central to the study of the  maximal pro-$p$ quotient  $\G_S(p)$ of $\G_S$, in particular when $S$ is coprime to $p$: obviously, 
one gets 
$$d_p H^2(\G_S(p),\fq_p) \leq  \sum_{v \in S} d_p H^2(\G_v,\fq_p)  + d_p \Sha_S 
\leq \sum_{v \in S}\delta_{v,p} + d_p \Sha_S \leq |S|+ d_p  \V_S/\K^{\times p},$$which is sufficient to produce criteria involving the infinitess of $\G_S(p)$ (thanks to the Golod-Shafarevich Theorem).
Here $\delta_{v,p}=1$ or $0$ as  $\K_v$ contains the $p$th roots of unity or does not. 

\medskip

Observe that thanks to (\ref{equation:0}), one can  force $\Sha_S$ to be trivial (see the notion of saturated set $S$ in \S \ref{section:saturated}), which can  also yield  situations where $\G_S(p)$ has cohomological dimension~$2$. See \cite{Labute} for the first examples and \cite{Schmidt} for general statements.

\medskip

Before giving our  main result, we make the following observation:  given $p$ a prime number, and two finite sets $S$ and $X$ of places of $\K$, one has: 
\begin{eqnarray}\label{equation:1}
\Sha_{S\cup X,p} \hookrightarrow \Sha_{S\cup X} \hookrightarrow \CyB_{S\cup X} \twoheadleftarrow \CyB_X \hookleftarrow \Sha_X  \hookleftarrow \Sha_{X,p}\end{eqnarray}
where the middle surjection follows as $\V_{S\cup X} \subset \V_{X}$.
To simplify, we consider only the case where the finite places $X$ and $S$ are coprime to $p$.
Here we prove:
 
\begin{Theorem} \label{theo:main}
Let $p$ be a prime number, and let $\K$ be a number field. 
Let $X$ be a finite set of  places  of $\K$ coprime to $p$.
There exist infinitely many  finite sets $S$ of finite places of $\K$, all coprime to $p$, such that:
$$\Sha_{S \cup X,p} \simeq \Sha_{S\cup X} \simeq \CyB_{S \cup X} \simeq \CyB_X \cdot$$
Moreover  such $S$ can be chosen of size $|S| \leq d_p\CyB_\emptyset$. 
\end{Theorem}

\medskip

Set $m:=d_p \CyB_\emptyset$. 
Note $\K^{ \times p} \subset \V_S$ for all $S$.
In particular, we have the exact sequence
$$ 0 \rightarrow \O^\times_\K/\O_\K^{\times p} \rightarrow \V_\emptyset/\K^{ \times p} \rightarrow Cl_\K [p] \rightarrow 0$$
 so $m=d_p \Cl_\K+ d_p\O_\K^\times$.

\medskip

As mentioned above, the computation of  $\Sha_S$ is very difficult in the tame case. 
Indeed, the only examples we know of where the map $\Sha_{\emptyset,p} \hookrightarrow \CyB_\emptyset$ is {\it not} an isomorphism are
those in which we know the relation rank of $G_\emptyset(p)$ by knowing the full group itself.
Using Theorem \ref{theo:main}, one may give situations where the value of $|\Sha_S|$ is known without being trivial.  As corollary, we get

\begin{Corollary} \label{coro2} 
There exist infinitely many  finite sets $S_0 \subset S_1 \subset \cdots \subset  S_m$  of finite places of $\K$ all coprime to $p$, such that for $i=0,\cdots, m$, one has
$$\Sha_{S_i,p} \simeq \Sha_{S_i} \simeq \fq_p^{m-i}.$$ 
\end{Corollary}

\medskip

\begin{Remark}
We will see that the sets $S$ and  $S_i$ can be explicitly given by the Chebotarev density Theorem in some governing field extension over $\K$.
\end{Remark}

\begin{Remark}
Let $\K_S^{ta}/\K$ be the maximal Galois extension of $\K$, unramified outside $S$, and tamely ramified at $S$; put $\G_S^{ta}=\Gal(\K_S^{ta}/\K)$. Then instead of considering $\G_S$ one may consider $\G_S^{ta}$ which also surjects onto $\G_S(p)$. Observe here that $\G_S^{ta}$ may be finite (typically when the discriminant of $\K$ and the norm of prime ideals of $S$ are too small), even trivial (for example when $\K=\Q$ and $S=\emptyset$).
\end{Remark}

{\bf Notations}

$-$ We fix a prime number $p$ and a number field $\K$. 

$-$ Put $\K'=\K(\zeta_p)$ and $\K''=\K(\zeta_{p^2})$, where $\zeta_{p^2}$ is some primitive $p^2$th root of unity, and $\zeta_p=\zeta_{p^2}^p$.

$-$ We denote by $\O_\K$ the ring of integers of $\K$, by $\O_\K^\times$ the group of units of $\O_\K$, and by $\Cl_\K$ the class group of $\K$.
 
$-$   We identify a prime ideal $\p \subset \O_\K$ with the place $v$ it defines. We write $\K_v$ for the completion of 
$\K$ at $v$ and $\U_v$ for the units of the local field $\K_v$; when $v$ is  archimedean, put $\U_v=\K_v^\times$. 
 
$-$ One says that a prime ideal  $\p$ is {\it tame} if  $\# \O_\K/\p \equiv 1 ({\rm mod} \ p)$, which is equivalent to 
$\mu_p \subset \K_v$, that is  $\delta_{v,p}=1$.

$-$ If $S$ is a finite set of places of  $\K$, we denote by $\K_S(p)/\K$ (resp. $\K_S^{ab}(p)/\K$) the maximal pro-$p$ extension (resp. abelian) of $\K$ unramified outside $S$, and we put $\G_S(p)=\Gal(\K_S(p)/\K)$ (resp. $\G_S^{ab}(p)=\Gal(\K_S^{ab}(p)/\K)$). For $S=\emptyset$,  we denote by  $\H:=\K_\emptyset^{ab}(p)$ the Hilbert $p$-class field of $\K$.

$-$  By convention, the infinite places in $S$ are only real. Let us write $S=S_0 \cup S_\infty$, where $S_0$ contains only the finite places and $S_0$  only the real ones. Put $\delta_{2,p}=  \left\{ \begin{array}{cc} 1 & p=2 \\ 0 & \mbox{otherwise} \end{array}\right.$

 $-$ The set $S$ is said to be coprime to $p$, if all finite places $v $  of $S$ are coprime to $p$; it is said to be tame if $S$ is coprime to $p$ and $S_\infty=\emptyset$.

$-$ Put $\V_S=\{ x \in \K^\times, v(x)\equiv 0 \  ({\rm mod} \ p) \ \forall v; x \in \K_v^p \ \forall v\in S\}$. Note $\K^{ \times p} \subset \V_S$ for all $S$.

\medskip

\section{Preliminaries}

\subsection{Extensions with prescribed ramification} \label{section:cyclicdegreep2}
Let $p$ be a prime number.

\subsubsection{Governing fields}
We recall a result of Gras-Munnier  
 (see \cite[Chapter V, \S 2, Corollary 2.4.2]{gras},  as well as \cite{Gras-Munnier}) which gives a criterion for the existence of totally ramified $p$-extension at some set $S$ (and unramified outside $S$).
Put  $\K':=\K(\zeta_p)$ and  consider the governing field $\L':=\K'(\sqrt[p]{\V_\emptyset})$. The extension $\L'/\K'$ 
has Galois group isomorphic to $(\Z/p\Z)^{r_1+r_2-1+\delta+d}$, where $d=d_p \Cl_\K$.   
 
 \medskip
 
 Given a place $v$ of $\K$, we choose some place $w|v$ of $\L'$ above $v$, and we consider $\sigma_v \in \Gal(\L'/\K')$ defined as follows:
 \begin{enumerate}
 \item[$-$]   if $v$ corresponds to a prime ideal $\p$ coprime to $p$, and $\P$ to $w$, then $\P$ is unramified in $\L'/\K'$, and then  
$\displaystyle{\sigma_v=\sigma_\p= \FF{\L'/\K'}{\P}}$  corresponds to the Frobenius elements at $\P$
in $\Gal(\L'/\K')$;
\item[$-$] if $v$ corresponds to a real place, then $\sigma_v$ is  the Artin symbol at $w$: $\sigma_v(\sqrt{\varepsilon})=+1$ is $\varepsilon$ is positive at $w$, and $-1$ otherwise.
\end{enumerate}
While $\sigma_v$ does in fact depend on the choice of $\P$, it is easy to see a different choice of $\P$ gives a nonzero multiple 
of the previous choice of $\sigma_v$ in the $\fq_p$-vector space $\Gal(\L'/\K')$.
This is all we need when invoking Theorem~\ref{theo:Gras-Munnier} below.
By abuse,  we will  also call the  $\sigma_v$'s  Frobenius elements.

\begin{theo}[Gras-Munnier] \label{theo:Gras-Munnier}  
Let $S=\{v_1,\cdots, v_t\}$ be a set of  places of $\K$ coprime to $p$. There exists a cyclic degree $p$ extension $\L/\K$, unramified outside $S$ and totally ramified at 
each place of $S$, if and only if, for  $i=1,\cdots, t$, there exists $a_i \in \fq_p^\times$, such that
$$\prod_{i=1}^t \sigma_{v_i}^{a_i} =1 \ \in \Gal(\L'/\K').$$
\end{theo}
 
When  $S$ is as  in the Gras-Munnier criterion, i.e. the necessary and sufficient condition of the theorem holds, one says that the elements $\sigma_{v_i}$  satisfy a {\it  strongly nontrivial relation}.  When one only has $\prod_{i=1}^t\sigma_{v_i}^{a_i} =1$, with the $a_i$ not all zero, one says that the  $\sigma_{v_i}$'s satisfy a {\it nontrivial relation}.

\begin{rema}
In fact, we don't find Theorem \ref{theo:Gras-Munnier} in \cite{gras} in this form, the difference coming from the real places (and then only for  $p=2$). Indeed, one starts with the following: for a real place $v$, in our context we speak of {\it ramification}, and in the context of \cite{gras} Gras speaks of {\it decomposition}.
Hence the governing field in \cite{gras} is smaller than $\L'$ and the condition he obtains did not involve the $\sigma_v$'s, $v\in S_\infty$ (in fact, in his case these $\sigma_v$ are trivial).  But the proof is the same, we can follow it without difficulty due to the fact that for $v\in S_\infty$, one has: $\U_v/\U_v^2=\Reel^\times/\Reel^{\times 2} \simeq \Z/2\Z$; see Lemmas 2.3.1, 2.3.2, 2.3.4 and 2.3.5 of \cite{gras}. 
\end{rema}

\begin{rema} \label{remark:frobenius}
The results of  \cite{gras} allow us to also obtain  the following: put $\L_0':=\K'(\sqrt[p]{\O_\K^\times})$, then 
$\#\G_S^{ab}(p) > \# \G_\emptyset^{ab}(p)$ if and only if, there exists some nontrivial relation  in $\Gal(\L_0'/\K')$ between  the $\sigma_v$'s, $v \in S$.
See also Proposition \ref{key-proposition0}.
\end{rema}

As consequence of Theorem \ref{theo:Gras-Munnier}, one has:

\begin{coro} \label{coro_grascriteria}
Given $p$ and $\K$, and two finite sets $T$ and $S$  of  places of $\K$ coprime to $p$, there exists a cyclic degree $p$ extension $\L/\K$, unramified outside $S\cup T$ and  ramified at each place of $S$ 
(no condition on the places of  $T$),
if and only if the $\sigma_v$'s for $v \in S$ satisfy a strongly nontrivial relation in the quotient $\Gal(\L'/\K')/\langle \sigma_v,  v \in T\rangle$. 
\end{coro}

\subsubsection{Extensions over the Hilbert $p$-class field  of $\K$ that are abelian over $\K$.}

As noted in the beginning of Chapter V of \cite{gras}, the result about the existence of a degree-$p^e$ cyclic extension with prescribed ramification can be generalized in different forms. 
Let $\H$ be the Hilbert class field of $\K$. In what follows, we only need  the existence of a degree-$p^2$ cyclic extension of $\H$, abelian over $\K$, with prescribed ramification.

\smallskip

Now we follow the  strategy of \cite[Chapter V, \S2, d)]{gras}. Put $B=\Gal(\K_S^{ab}(p)/\H)$. Take $\Sigma$ a finite set of places of $\K$ coprime to $p$ (not necessarily satisfying the congruence  $\N(\p) \equiv 1 ({\rm mod } \ p^2)$ when $\p \in \Sigma_0$). By class field theory, we get
$$1 \longrightarrow (B/B^{p^2})^* \stackrel{\rho}{\longrightarrow} \bigoplus_{v\in \Sigma} (\U_v/(\U_v)^{p^2})^* \longrightarrow \big(\iota (\O_\K^\times)\big)^* \longrightarrow 1,$$
where $\displaystyle{\iota: \O_\K^\times \longrightarrow  \bigoplus_{v \in \Sigma} \U_v/(\U_v)^{p^2}}$ is the diagonal embedding.

\medskip

 A cyclic degree-$p^2$ extension $\M$ of 
$\H$, abelian over $\K$ and 
unramified outside $\Sigma$ is given by a character~$\psi$ of~$B/B^{p^2}$ of order  $p^2$ as follows: 

Given $\psi_v \in (\U_v/(\U_v)^{p^2})^*$ for all $v\in \Sigma$, there exists a character $\psi$ of  $B/B^{p^2}$ such that $\psi_{|\U_v}=\psi_v$ if and only if, 
\begin{eqnarray}\label{equation:2} \forall \varepsilon \in \O_\K^\times, \ \prod_{v \in \Sigma} \psi_v(\varepsilon)=1.\end{eqnarray}

As $\M/\H$ is totally ramified at at least one prime ideal,   at least one $\psi_v$ has order $p^2$.


\medskip

 Now we will focus on the case where $\Sigma$ contains only finite places, and we use the notation $\p$ instead of $v$.
 
 \medskip
 
Let $S$ be a finite non-empty set of tame places of $\K$ where each prime $\p$ (corresponding to $v \in S$) is such that  $\N(\q) \equiv 1 ({\rm mod } \ p^2)$.
 Let us write now $\Sigma_\q=S  \cup T_\q$, where $T_\q=\{\q\}$ is also tame. We are interested in the existence of a degree-$p^2$ cyclic extension $\K_\q/\H$, abelian over $\K$ and
 unramified outside $\Sigma_\q$, such that $\K_\q/\H$ has degree $p^2$ and 
for which the inertia degree at $\q$ is exactly $p$.

For $\p \in \Sigma_\q$, let us fix $\chi_\p$ a generator of  $(\U_\p/(\U_\p)^{p^2})^*$.
By (\ref{equation:2}), $\K_\q$  exists if and only if, there exist $a_\q \in \fq_p^\times$, and $b_\p \in \Z/p^2$, $\p \in S$, such that 
$$\forall \varepsilon \in \O_\K^\times, \  {\hat \chi_\q}^{a_\q}(\varepsilon) \prod_{\p \in S}   \chi_\p^{b_\p}(\varepsilon)=1,$$
where $${\hat \chi_\q}=\left\{\begin{array}{ll} 
\chi_\q & {\rm if} \ \N(\q) \nequiv 1 ({\rm mod } \ p^2) \\
\chi_\q^p &   {\rm if} \ \N(\q) \equiv 1 ({\rm mod } \ p^2)
\end{array}
\right. ,$$
and such that at least one $b_\p \in (\Z/p^2\Z)^\times$.
 
This last condition can be rephrased thanks to Kummer theory with the  following governing field (see \cite[Chapter V, \S 2, d)]{gras}): $$\L=\K''(\sqrt[p^2]{\O_\K^\times}),$$ where $\K''=\K(\zeta_{p^2})$. 
For each prime $\p \in \Sigma_\q$ let us choose a prime $\P|\p$ of $\K''$, and denote by $\sigma_\p$ the Frobenius of $\P$ in $\Gal(\L/\K'')$. 
As before,  $\sigma_\p$  depends on $\P|\p$ only up to a power coprime to $p$.

The above discussion allows us to obtain the following:

\begin{prop} \label{key-proposition0}
There exists a degree-$p^2$ cyclic extension $\K_\q/\H$, abelian over $\K$, unramified outside $\Sigma_\q$,  
for which the inertia degree at $\q$ is exactly $p$, if and only if, there exists 
$a_\q \in \fq_p^\times$, and  $b_\p\in \Z/p^2\Z$, $\p \in S$, such that 
\begin{eqnarray} \label{equation:3} {\hat \sigma}_\q^{a_\q} \prod_{\p \in S}   \sigma_\p^{b_\p}   =1 \in \Gal(\L/\K''),
\end{eqnarray}
where $${\hat \sigma_\q}=\left\{\begin{array}{ll} 
\sigma_\q & {\rm if} \ \N(\q) \nequiv 1 ({\rm mod } \ p^2) \\
\sigma_\q^p &   {\rm if} \ \N(\q) \equiv 1 ({\rm mod } \ p^2)
\end{array}
\right. ,$$
with at least one $b_\p \in (\Z/p^2\Z)^\times$.
\end{prop}

\begin{Remark}
Infinitely many such sets exist by the Chebotarev Density Theorem. 
\end{Remark}

\subsection{Saturated sets} \label{section:saturated}

Take $p$, $\K$ as before, and let $S$ be a finite set of places of $\K$, coprime to $p$.

\begin{defi}
The $S$  set of places $\K$ is called saturated if $\V_S/(\K^\times)^p=\{1\}$. 
\end{defi}

Recall the following equality due to Shafarevich (see for example \cite[Chapter X, \S 7, Corollary 10.7.7]{NSW}):
\begin{eqnarray} \label{egalite:p-rank} d_p \G_S=|S_0|+|S_\infty| \delta_{2,p} -(r_1+r_2)+1-\delta  + d_p \V_S/(\K^\times)^p,
\end{eqnarray}
showing that $d_p \G_S$ is easy to compute when $S$ is saturated.

\begin{prop} \label{proposition1.6}
Let $S$  and $T$ be two finite sets of places of $\K$ coprime to $p$. Suppose $S$ is saturated. Then 
\begin{enumerate}
\item[$-$] if  $S\subset T$,  then $T$ is saturated;
\item[$-$] for every  tame place $v \notin S$, one has $d_p \G_{S\cup \{\p\}}=d_p \G_S +1$.
\end{enumerate}
\end{prop}

\begin{proof}
The first point is  due to the fact that $\V_T \subset \V_S$, and the second point is  a consequence of (\ref{egalite:p-rank}) along with the first point.
\end{proof}

\begin{theo} \label{theo_saturated}
A finite set  $S$ coprime to $p$ is saturated if and only if, the Frobenii $\sigma_v$, $v \in S$, generate the whole group $\Gal(\K'(\sqrt[p]{\V_\emptyset})/\K')$.
\end{theo}

\begin{proof}
$\bullet$  
 Suppose the Frobenii generate the full Galois group. 
By hypothesis, for each degree-$p$ extension  $\L/\K'$ in 
 $\K'(\sqrt[p]{\V_\emptyset})/\K'$, 
there exists a place  $v\in S$ such that $v$ is inert in $\L/\K'$ (when $v\in S_\infty$, $v$ is ramified in $\L/\K'$).  Let us take  now $x\in \V_S$: then every $v \in S$ splits totally in $\K'(\sqrt[p]{x})/\K'$. As $\K'(\sqrt[p]{x}) \subset \K'(\sqrt[p]{\V_\emptyset})$, one deduces that $\K'(\sqrt[p]{x})=\K'$, and then $x\in (\K')^p$. As $[\K':\K]$ is coprime to $p$, one finally obtains  that $x\in \K^{\times p}$, so  $\CyB_S=\{0\}$.

$\bullet$ If $S$ is saturated, then for every finite set $T$ of tame places of $\K$ with $T\cap S=\emptyset$, one has $d_p \G_{S\cup T}=d_p \G_S + |T|$ by Proposition \ref{proposition1.6}. Then by the Gras-Munnier criterion, one has $\langle \sigma_v , v \in S \rangle = \Gal(\L'/\K')$.
\end{proof}

\begin{coro}
The finite  set $S$ coprime to $p$ is saturated if and only if, for every finite set $T$ of tame places of $\K$, there exists a cyclic degree $p$-extension 
 of $\K$ 
unramified outside $S\cup T$ but  ramified at each place of $T$.
\end{coro}

\begin{proof}
$\bullet$ If $S$ is saturated, then by Theorem \ref{theo_saturated} the Frobenii $\sigma_v$, $v \in S$, generate $\Gal(\L'/\K')$, and the result follows
by using Corollary \ref{coro_grascriteria}.

$\bullet$ 
Suppose that $S$ is such that for every finite set $T$ of tame places of $\K$, there exists a cyclic degree $p$-extension unramified outside $S\cup T$ and  ramified at each place of $T$. Then by Corollary \ref{coro_grascriteria} and the Chebotarev density theorem, $\Gal(\L'/\K')=\langle \sigma_v, \ v \in S\rangle$. By Theorem \ref{theo_saturated}, $S$ is saturated.
\end{proof}

\subsection{Spectral sequence}

Let $S$ and $T$ be two finite sets of places of $\K$ coprime to~$p$. Consider the following exact sequence of pro-$p$ groups
\begin{eqnarray} \label{equation:se}
1 \longrightarrow \H_{S,T} \longrightarrow \G_{S\cup T}(p) \longrightarrow \G_S(p) \longrightarrow 1.
\end{eqnarray}

\begin{defi} Put
$$\XX_{S,T}:=\H_{S,T}/[\H_{S,T},\H_{S,T}]\H_{S,T}^p,$$
and
$$\X_{S,T}:=\left(\XX_{S,T}\right)_{\G_S(p)}=\H_{S,T}/ [\H_{S,T},\G_S(p)]\H_{S,T}^p.$$
\end{defi}

Recall that as $\G_S(p)$ is a pro-$p$ group, then  $\fq_p\ldbrack \G_S(p) \rdbrack$ is a local ring.

\begin{lemm}
The abelian group $\XX_{S,T}$ is a $\fq_p\ldbrack \G_S(p) \rdbrack$-module  (with continuous action) that can be generated by $d_p \X_{S,T}$ generators.
Moreover, $d_p \X_{S,T} \leq |T|$.
\end{lemm}

\begin{proof}
The first part follows from Nakayama's lemma. For the second, 
the fact that $\G_S(p)$ acts transitively on the inertia groups $I_w$ of $w|v \in T$ in $\XX(S,T)$ implies
 $$\bigoplus_{i=1}^t \fq_p\ldbrack \G_S(p) \rdbrack \twoheadrightarrow \langle I_w, w|v \in T \rangle = \XX_{S,T},$$
where $t=|T|$. Taking the $\G_S(p)$-coinvariants, we obtain $\fq_p^t \twoheadrightarrow \X_{S,T}$.
\end{proof}

Applying the Hochschild-Serre spectral sequence to (\ref{equation:se}), one gets: 
\begin{lemm} \label{chase}Let $ S, T$ be two finite sets of  places of $\K$ coprime to $p$.
Then one has~: \label{exact-sequence1}
$$ 
 1 \longrightarrow H^1(\G_S(p),\fq_p) \longrightarrow H^1(\G_{S\cup T}(p),\fq_p) \longrightarrow \X_{S ,T}^\vee 
\longrightarrow{ \Sha_{S,p}} \longrightarrow {\Sha_{ S\cup T,p} }.
$$
Furthermore, the cokernel of the natural injection $\sha_{X,p} \hookrightarrow \CyB_X$ decreases in dimension as $X$ increases. 
\end{lemm} 

\begin{proof} 
The Hochschild-Serre spectral sequence gives the exact commutative diagram:

\xymatrix{
 H^1(\G_S(p),\fq_p)  \ar@{^{(}->}[r] &  H^1(\G_{S\cup T}(p),\fq_p)  \ar@{->}[r] & \X_{S,T}^\vee  \ar@{->}[r] &  H^2(\G_S(p),\fq_p)  \ar@{->}[r] \ar@{->}[d]  & 
 H^2(\G_{S\cup T}(p),\fq_p)  \ar@{->}[d] \\
 & && \oplus_{v\in S} H^2(\G_v,\fq_p) \ar@{^(->}[r]& \oplus_{v \in S\cup T} H^2(\G_v,\fq_p)}
Chasing the trangression map $\X^\vee_{S,T} \stackrel{tg}{\longrightarrow} H^2(G_S(p))$ to the right gives that its image lies in $\Sha_{S,p}$ whose image to the right
lies in $\Sha_{S\cup T,p}$.
We now have the diagram

\xymatrix{
 1 \ar@{->}[r] & H^1(\G_S(p),\fq_p)  \ar@{->}[r] &  H^1(\G_{S\cup T}(p),\fq_p)  \ar@{->}[r] & \X_{S,T}^\vee  \ar@{->}[r] &  \Sha_{S,p}  \ar@{->}[r] \ar@{^(->}[d]  & \Sha_{S\cup T,p}  \ar@{^(->}[d] \\
& & && \CyB_S \ar@{->>}[r]& \CyB_{S\cup T} }
where the bottom horizontal map is surjective as the inclusion $\V_{S\cup T}/\K^{\times p} \hookrightarrow \V_S/\K^{\times p}$ is immediate from the definition of $\V_X$.
The second result follows.
\end{proof}

\begin{coro}
If the natural injection $\sha_{X,p} \hookrightarrow \CyB_X$ is an ismorphism, then for any set $Y$ we have
$\sha_{X\cup Y,p} \stackrel{\simeq}{\hookrightarrow} \CyB_{X \cup Y}$
\end{coro}

Let us give an  obvious consequence of Lemma~\ref{exact-sequence1}.

\begin{lemm} \label{lemma1} Suppose that $H^1(\G_S(p),\fq_p) \simeq H^1(\G_{S\cup T}(p),\fq_p)$, then $\X_{S,T}^\vee \hookrightarrow \Sha_{S,p}$. If moreover  $ S \cup T$ is saturated then
 $\X_{S,T}^\vee \simeq \Sha_{S,p}$.
\end{lemm}

\begin{proof} If $S\cup T$ is saturated then $\V_{S\cup T}/\K^{\times p}=\{1\}$, which implies that $\CyB_{S\cup T}=\{1\}$.
Hence, by (\ref{equation:0}) $\Sha_{S\cup T}=\{0\}$, and the same holds for $\Sha_{S\cup T,p}$. We conclude with Lemma \ref{exact-sequence1}.
\end{proof}

An important consequence of Lemmas~\ref{chase} and~\ref{lemma1} is that elements of $\X_{S,T}^\vee$ can give rise to elements of $\Sha_{S,p}$. The former can be found 
via ray class group computations. We thus have a method of producing independent elements of $\Sha_{S,p}$. If we find $d_p \CyB_S$ such elements, we have established
$\Sha_{S,p} \stackrel{\simeq}{\hookrightarrow} \Sha_S  \stackrel{\simeq}{\hookrightarrow} \CyB_S$, and thus computed $d_p \Sha_S$.


\section{Proof of the results}

 \subsection{A key Proposition}
 Let $p$ be a prime number. Let $\K$ be a number field and let $X$ be a finite set of places of $\K$ coprime to $p$.
The proof of Theorem \ref{theo:Gras-Munnier} is a   consequence of the following proposition.

\begin{prop} \label{prop:ramification-in-p2}
There exist (infinitely many) pairs of finite sets  of tame places~$S$ and~$T$ of~$\K$ such   that:
\begin{enumerate}[$(i)$]
\item $T \cup X$ is saturated and $d_p \G_{T\cup X}=d_p \G_X$;
\item  $d_p \G_{S\cup T \cup X}=d_p \G_{S \cup X}$;
\item  $|T|  \leq d_p  \Cl_\K+ r_1+r_2-1+\delta $  and $|S|  \leq r_1+r_2-1+\delta$; 
\item for each prime $\q \in T$, there exists a degree-$p^2$ cyclic extension $\K_\q$ of $\K^H$, abelian over~$\K$, unramified outside $S\cup X \cup  \{\q\}$, where the inertia group at $\q$ is of order  $p$.  
\end{enumerate}
\end{prop}

\medskip

Put $\F_0=\K'(\sqrt[p]{\V_\emptyset})$, $\L_0=\K'(\sqrt[p]{\O_\K^\times})$,  $\K''=\K(\zeta_{p^2})$,  $\L_1=\K''(\sqrt[p^2]{\O_\K^\times})$, $\F_1=\K''(\sqrt[p]{\V_\emptyset})$, and $\F=\L \F_0=\K''(\sqrt[p^2]{\O_\K^\times}, \sqrt[p]{\V_\emptyset})$.  
Put $\G=\Gal(\F/\K')$.

\begin{proof} (of Proposition \ref{prop:ramification-in-p2}.)

Given a prime $\p $ of $\O_\K$, coprime to $p$,  we choose a prime $\P|\p$ of $\F$, and we  consider its Frobenius $\sigma_\p:=\sigma_\P$ in the Galois group $\Gal(\F/\K')$  and  its quotients. As mentioned earlier, this is well-defined up to a nonzero scalar multiple in $\Gal(\F/\K')$ and that is all we need.

\medskip

Put $E_X=\langle {\sigma_\p}_{|\F_0}, \p \in X\rangle \subset \Gal(\F_0/\K')$ the subgroup of $\Gal(\F_0/\K')$ generated by the Frobenii of the primes $\p \in X$. Put $m_\X=d_p \V_\emptyset- d_p E_X$.

\medskip

a) \underline{Assume first that $\F_0\cap \K''= \K'$.}

{\tiny
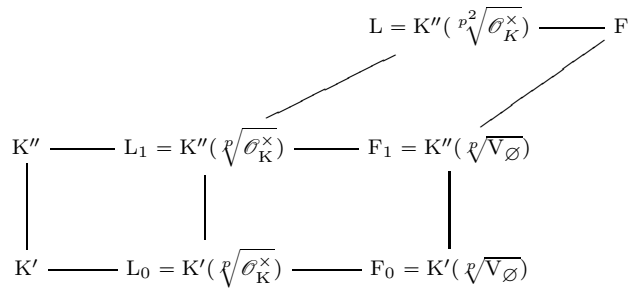
$$\xymatrix{ 
& & \L=\K''(\sqrt[p^2]{\O_K^\times}) \ar@{-}[r]   & \F  \ar@{-}[ld] \\
\K'' \ar@{-}[r]&\L_1=\K''(\sqrt[p]{\O_\K^\times}) \ar@{-}[ru] \ar@{-}[r]& \F_1=\K''(\sqrt[p]{\V_\emptyset})&\\
\K'   \ar@{-}[u] \ar@{-}[r]& \L_0=\K'(\sqrt[p]{\O_\K^\times}) \ar@{-}[u] \ar@{-}[r] &\ar@{-}[u] \F_0=\K'(\sqrt[p]{\V_\emptyset})&
}$$
}

We choose $S$ and $T$ as follows: 
\begin{enumerate}
\item[$-$] let $T$ be {\it any} set of primes  $\q$ whose Frobenii $\sigma_\q$  in $\G$  are such that the restriction in $\Gal(\F_0/\K')$ forms  an $\fq_p$-basis of a subspace  in direct sum with $E_X$: in other words, $$\displaystyle{\Gal(\F_0/\K')=\langle {\sigma_\q}_{|\F_0}, \q \in T \rangle \bigoplus E_X},$$ and $\displaystyle{ \langle {\sigma_\q}_{|\F_0}, \q \in T \rangle = \bigoplus_{\q \in T} \langle  {\sigma_\q}_{|\F_0}\rangle}$. 
\item[$-$]   let $\tilde{X}$ be those places of $X$ whose Frobenii lie in $\Gal(\F/\F_1)$ and
let $S$ be {\it any} set of primes  $\p$ whose Frobenii $\sigma_\p$  in $\G$ form in direct sum with the Frobenii in $\tilde{X}$ a basis of $\Gal(\F/\F_1)$.
\end{enumerate} 

\smallskip
 
As $\Gal(\F_1/\K')$  has exponent   $p$,  we see for each $\q \in T$, $\sigma_\q^p \in \Gal(\F/\F_1)$. Observe also that if ${\sigma_\q}_{|\K''}$ is not trivial (which is equivalent to $\N(\q) \neq 1 \ ({\rm mod} \ p^2)$), then $\sigma_\q^p$ is the Frobenius at $\P$ in $\Gal(\F/\F'')$; otherwise $\sigma_\q^p$ is the $p$-power of the Frobenius at 
$\mathfrak Q \mid \q$ in $\Gal(\F/\F'')$.

\smallskip

By Theorem \ref{theo_saturated} the set $T\cup X$ is saturated. 
Moreover thanks to the condition on the direct sum for the Frobenius at $\p \in T$,   by  Theorem \ref{theo:Gras-Munnier}, there is no cyclic degree-$p$ extension of $\K$,  unramified outside $T \cup X$ and totally ramified at any nonempty subset of places of~$T$: thus $d_p \G_{T\cup X}=d_p \G_X$, and $(i)$ holds.

\smallskip

Moreover as each place of $S$ splits totally in the governing extension $\F_0/\K'$, then again by Theorem \ref{theo:Gras-Munnier}, $d_p \G_{S\cup T\cup X}=d_p \G_{S\cup X}$, and $(ii)$ holds.

\smallskip
 
 The condition on $S$ gives a relation of type  (\ref{equation:3}) in $\Gal(\F/\F_1) \subset \Gal(\F/\L_1)$
 for the set $S\cup \tilde{X}\cup \{\q\}$, $\q \in T$.
 After taking the 
 quotient of this relation by
  $\Gal(\F/\L)$, we obtain
by Proposition \ref{key-proposition0} that for each prime $\q \in T$, the existence of a degree-$p^2$ cyclic extension $\K_\q/\H$, abelian over $\K$ and unramified outside $S\cup X \cup \{\q\}$ for
  which the inertia at $\q$ is of order $p$, proving $(iv)$.

  \smallskip
  
  $(iii)$ is obvious.

\medskip

b)  \underline{Assume now that that $\K''\subset \F_0$}.

Let $\A_i, i=1,\cdots, d$ be  ideals of $\O_\K$,  whose classes are a system of minimal generators of $\Cl_\K[p]$, and let $a_i \in \O_\K^\times$ such that $(a_i)=\A^p_i$.  Put $A=\langle a_1,\cdots, a_d \rangle \K^{\times p}/\K^{\times p}   \subset \V_\emptyset/\K^{\times p}$.
Note $\K'(\sqrt[p]{\V_\emptyset})=\K'(\sqrt[p]{A},\sqrt[p]{\O_\K^\times})$. 

As $\F_0/\K'$ and $\K''/\K'$ are abelian $p$-extensions, the containment $\K''\subset \F_0$ implies $\K'=\K$.
Moreover $\L_0 \cap \K''(\sqrt[p]{A})=\K''$.
{\tiny
$$\xymatrix{& \L=\K''(\sqrt[p^2]{\O_\K^\times})  \ar@{-}[r]& \F \\
\L_0=\K'(\sqrt[p]{\O_\K^\times}) \ar@{-}[ur] \ar@{-}[r] & \F_0=\F_1 =\K'(\sqrt[p]{\V_\emptyset})  \ar@{-}[ur]  &   \\
 \K'' \ar@{-}[u]  \ar@{-}[r]& \K''(\sqrt[p]{A})  \ar@{-}[u]   & \\
  \K' \ar@{-}[u] \ar@{-}[r] & \K'(\sqrt[p]{A})\ar@{-}[u]& 
}
$$
}

Now take $T$ and $S$ as in case a). 
\end{proof}

\begin{rema}
Observe that one can take $T$ such that $|T|\leq m_\X= d_p \V_\emptyset -d_p E_X$.
\end{rema}

\subsection{Proof of Theorem \ref{theo:main} }
Let $S$ and $T$ as in Proposition \ref{prop:ramification-in-p2}.
As $X\cup T$ is saturated, by $(i)$ of Proposition \ref{prop:ramification-in-p2} and (\ref{egalite:p-rank}),  one obtains  $|T|=d_p \CyB_X$.
Moreover,  $S\cup X \cup T$ is also saturated and in particular, $\CyB_{S\cup X\cup T} \simeq  \Sha_{S\cup X \cup T,p} =\{0\}$.
With $(ii)$, we see  that $\d_p \CyB_{S\cup X}=|T|$ so $(i)$ and $(ii)$ imply:  $\CyB_{S\cup X} \simeq \CyB_X$.

\smallskip

Now let us take  the spectral sequence of the short exact sequence $$1 \longrightarrow \H_{S\cup X,T} \longrightarrow \G_{S\cup X\cup T}(p) \longrightarrow \G_{S \cup X}(p) \longrightarrow 1$$ to obtain by Lemma \ref{exact-sequence1}:
$$1 \rightarrow H^1(\G_{S\cup X}(p),\fq_p) \rightarrow H^1(\G_{S\cup X\cup T}(p),\fq_p) \rightarrow \X_{S\cup X,T}^\vee \rightarrow \Sha_{S \cup X,p} \rightarrow \Sha_{S\cup X\cup T,p} =\{0\}.$$ 
Hence, $\X_{S\cup X,T}^\vee \simeq \Sha_{S\cup X,p}$.
Now $(iv)$ of Proposition \ref{prop:ramification-in-p2} implies that $d_p \X_{S\cup X,T} \geq |T|$, and as  obviously $d_p \X_{S\cup X,T} \leq |T|$, we finally get
$d_p\Sha_{S\cup X,p}=|T|$. 

Hence $d_p \Sha_{S\cup X,p}=|T|= d_p \CyB_{S\cup X} = d_p \CyB_X$. Thanks to (\ref{equation:1}), one has 
$$ \Sha_{S\cup X,p} \simeq \Sha_{S\cup X} \simeq \CyB_{S \cup X} \simeq \CyB_X .$$

\subsection{Proof of Corollary \ref{coro2}} Let us choose $S$ and $T$ as in proof of Proposition \ref{prop:ramification-in-p2}. Let us write $T=\{\p_1,\cdots, \p_{m_X}\}$, where $m_X=d_p \CyB_\emptyset-d_p E_X$. Put $S_0=S\cup X$ and, for $i \geq 0$,  $S_{i+1}=S\cup X\cup  \{\p_i\}$. Here, as $d_p \G_{S_i}=d_p \G_{S_{m_X}}$, the spectral sequence shows that 
\begin{eqnarray} \label{equality:diagramm}
\fq_p \hookrightarrow \Sha_{S_i,p} \longrightarrow \Sha_{S_{i+1},p},
\end{eqnarray}
 in particular $d_p \Sha_{S_{i},p} \leq d_p \Sha_{S_{i+1},p} +1$. After noting that  $d_p \Sha_{S_{m_X},p}=0$ (the set $X\cup T$ is saturated) and  that $d_p \Sha_{S_0,p}=|T|=m_X$, then we conclude that 
 $d_p \Sha_{S_i,p}=m_\X-i$. 
 Observe also that (\ref{equality:diagramm}) induces:
 $$\fq_p \hookrightarrow \Sha_{S_i} \longrightarrow \Sha_{S_{i+1}},$$ and as before $d_p \Sha_{S_i}=m-i$. The isomorphisms
 $\Sha_{S_i,p} \simeq \Sha_{S_i}$'s become obvious.
 
 We have proved:
 
 \begin{coro}
 One has $\Sha_{S_i} \simeq \fq_p^{m_X-i}$.
  \end{coro}

Take $X=\emptyset$ to have Corollary \ref{coro2}.

\section{Examples}
 
In this section we give a few examples of fields $\K$ and sets $S$
such that in the diagram 
$$\Sha_\emptyset \hookrightarrow \CyB_\emptyset \twoheadrightarrow \CyB_S \hookleftarrow \Sha_S,$$
the two maps on the right are isomorphisms.
In our first two examples we show 
the left map is {\it not} an isomorphism. Thus we give explicit examples where $\Sha_X$ increases as $X$ does, in contrast to the wild case.
\smallskip

In the third example we establish 
$$ \CyB_\emptyset \stackrel{\simeq}{\twoheadrightarrow} \CyB_S \stackrel{\simeq}{\hookleftarrow} \Sha_S,$$
but do not know whether $d_p \Sha_\emptyset < d_p \Sha_S$. Indeed, we suspect equality in that case.

\smallskip
In the examples below, $p_i$ refers to the $i$th prime of $\K$ above the rational prime $p$ as MAGMA presents the factorization. All code was run unconditionally, that is we did {\it not} use  GRH bounds for computing ray class groups.
 
 \begin{Example}
Let $\K$ be the unique degree $3$ subfield of $ \Q(\zeta_{7})$  and let $p=2$. Then one can easily compute that 
$\K$ has trivial class group and, since $\K$ is totally real,
$d_p \CyB_\emptyset= d_p \O^\times_\K /\O^{\times 2}_\K + d_p Cl_\K[2]=3$. 
Clearly $G_\emptyset = \{e\}$ and $d_p \Sha_\emptyset=0$ so $\Sha_\emptyset \hookrightarrow \CyB_\emptyset$ has $3$-dimensional cokernel.
Set $S = \{37_1, 181_1, 293_1\}$ and
$T=\{307_1,311_1,349_1\}$. 
One computes
$d_p H^1(G_T,\fq_2)=0$ so $T$ and $S\cup T$ are saturated.
The $2$-parts of the ray class groups for conductors  $S\cup T$ and $S$ are  $ (\Z/4)^3$ and $(\Z/2)^3$ respectively, so the 
the map $H^1(G_S,\fq_2) \rightarrow H^1(G_{S\cup T},\fq_2)$ is an isomorphism and 
$d_p \X_{S\cup X,T}^\vee \geq 3$.
As  $d_p \Sha_S \leq d_p \CyB_S  \leq  d_p \CyB_\emptyset =3$, we see $d_p \Sha_S=3$. 
\end{Example}

\begin{Example}
Let $\K$ be the unique degree $3$ subfield of $ \Q(\zeta_{349})$  and let $p=2$. 
Here $\K$ has class group $(\Z/2)^2$ and is again totally real, so
$d_p \CyB_\emptyset=  d_p \O^\times_\K /\O^{\times 2}_\K + d_p Cl_\K[2]=5$.
 One computes the class group of the Hilbert class field of $\K$ is trivial
 so $\G_\emptyset =\Z/2 \times \Z/2$ and  has three relations.
 Thus $d_p \Sha_\emptyset =d_p H^2(G_\emptyset,\fq_2)=3$ so the map $\Sha_\emptyset \hookrightarrow \CyB_\emptyset$ has $2$-dimensional cokernel.
Set $S = \{701_1, 2857_1, 3169_1\}$ and
$T=\{367_1,397_1,401_1,409_1,449_1\}$. One computes 
$d_p H^1(G_T,\fq_2)=2$ so $T$ and $S\cup T$ are saturated.
The $2$-parts of the ray class groups for conductors  $S\cup T$ and $S$ are  $ \Z/4 \times (\Z/8)^2\times \Z/16 \times \Z/32$ and $(\Z/2)^5$ respectively, so the 
the map $H^1(G_S,\fq_2) \rightarrow H^1(G_{S\cup T},\fq_2)$ is an isomorphism and 
$d_p \X_{S\cup X,T}^\vee \geq 5$.
As  $d_p \Sha_S \leq d_p \CyB_S  \leq  d_p \CyB_\emptyset =5$, we see $d_p \Sha_S=5$. 
\end{Example}
 
\begin{Example}
Let $\K=\Q[x]/(f(x))$ where $f(x)= x^{12}+339x^{10}-19752x^8-2188735x^6+284236829x^4+4401349506x^2+15622982921$. This polynomial is irreducible
and $\K$ is totally complex with small root discriminant and has class group $(\Z/2)^6$.  The field $\K$ has been used as a starting point in finding infinite towers of
totally complex number fields whose root discriminants are the smallest currently known.
Set $$S= \{7_2,11_1,43_1,47_3,67_3,97_1 \},\,\,
T= \{ 5_1, 13_1, 19_1,19_2,23_1, 23_2, 23_3, 29_1, 31_1, 61_1, 149_1,149_4\}.$$
As $\K$ is totally complex, $$d_p \CyB_\emptyset =  d_p \O^\times_\K /\O^{\times 2}_\K + d_p Cl_\K[2]= 6+6=12= \# T.$$
One computes  
$d_p H^1(G_T,\fq_2)=6$ so $T$ and $S\cup T$ are saturated.
The $2$-parts of the ray class groups for conductors  $S\cup T$ and $S$ are  
$(\Z/4)^5 \times (\Z/8)^4 \times (\Z/16)^3$ and
$(\Z/2)^{11} \times \Z/8$. 
respectively, so the 
the map $H^1(G_S,\fq_2) \rightarrow H^1(G_{S\cup T},\fq_2)$ is an isomorphism.
From this data one can only  conclude 
$d_p \X_{S\cup X,T}^\vee \geq 11$. On the other hand, for every $v\in T$ one computes the $2$-part of the ray class group for conductor $S \cup \{v\}$ has order
at least $2^{15} > 2^{14}$. As the latter quantity is the order of the $2$-part of the ray class group with conductor $S$, we  get $\# T=12$ independent elements
of $ \X_{S\cup X,T}^\vee$  so $d_p \Sha_S \geq 12$. As $d_p \CyB_S \leq d_p \CyB_\emptyset =12$, we have $d_p \Sha_S =12$.
{\it We suspect that in this case $d_p \Sha_\emptyset=12$.}

\end{Example}



\end{document}